\documentclass[preprint,review,12pt]{elsarticle}

\usepackage{amssymb}
\usepackage{amsmath}
\usepackage{gastex}
\usepackage{amsfonts}
\usepackage{amsmath}
\usepackage{amsthm}
\usepackage{amssymb}
\usepackage[all]{xy}
\usepackage[latin1]{inputenc}
\usepackage{graphicx}
\usepackage{color}

\input xy
\xyoption{all}

\newtheorem{theorem}{Theorem}[section]
\newtheorem{corollary}[theorem]{Corollary}
\newtheorem{lemma}[theorem]{Lemma}

\theoremstyle{definition}
\newtheorem{definition}[theorem]{Definition}
\newtheorem{remark}[theorem]{Remark}

\newcommand\iso{\kern.35em{\raise3pt\hbox{$\sim$}\kern-1.1em\to}\kern.3em}

\begin{document}

\begin{frontmatter}

\title{On the $K$-theory of feedback actions on linear systems\tnoteref{label2}}
\tnotetext[label2]{Some results of this work were presented as a talk at the meeting \textsl{Recent Trends in Rings and Algebras}, RTRA2013, Universidad de Murcia, Spain. June, 3rd, 2013.}
\author{Miguel V. Carriegos\tnoteref{label1}}
\tnotetext[label1]{Partially supported by Ministerio de Ciencia y Tecnología and INTECO. Ministerio de Industria, Spain.}
\ead{miguel.carriegos@unileon.es}
\author{ Ángel Luis Muñoz Castañeda}
\ead{u85212@usal.es}
\address{Departamento de Matem\'aticas. Universidad de Le\'on}

\begin{abstract}
A categorical approach to linear control systems is introduced. Feedback actions on linear control systems arises as a symmetric  monoidal category $S_R$. Stable feedback isomorphisms generalize dynamic enlargement of pairs of matrices. Subcategory of locally Brunovsky linear systems $B_{R}$ is studied and the stable feedback isomorphisms of locally Brunovsky linear systems are characterized by the Grothendieck group $K_{0}(B_R)$. Hence a link from linear dynamical systems to algebraic K-theory is stablished.
\end{abstract}


\end{frontmatter}

\section{Introduction}

This paper deals with the study of feedback actions on a linear control system. A concrete description of the feedback classification of constant linear systems by means of sets of invariants and canonical forms  goes back to the seminal works by Kalman, Casti and Brunovsky (see the fundamental references \cite{B}, \cite{K} and \cite{C}).

The more general framework of parametrized families of linear systems (see \cite{HS} or \cite{VW}) is proved to be a hard task (wild problem in the sense of representation theory  \cite{BK01}). Thus we need to restrict ourselves to the class of locally Brunovsky systems because a complete description of feedback invariants is available (see \cite{M}).

On the other hand we also are interested in the so-called dynamic feedback equivalence of linear systems (see \cite{BK}, \cite{HT}, \cite{HCT} as main references). This dynamic study is based in the addition of some suitable ancillary state variables to systems \cite[ch. 4]{BBV}. We introduce the notion of stable feedback equivalence and show that it is a generalization of both feedback and dynamic feedback equivalence. This generalization does not go too far because if the base ring is a field then feedback, dynamic feedback and stable feedback equivalence are proved to be the same notion. 

We think Category Theory is an adequate tool to study above subjects. First of all, the definition of the category $S_R$ of linear systems over a commutative ring and feedback actions arises in a natural way, more over feedback equivalences are the isomorphisms in the category. Then dynamic enlargements and stabilization of linear systems are consequence of some bi-product (both product and coproduct) in the category, hence the symmetric  monoidal structure of the category arises, and therefore we obtain that the stable feedback equivalences are the stable isomorphisms in the category. As a consequence the invariant characterizing the stable equivalence is the $K_0$ group of the $K$-theory of the category, which is just the Grothendieck group completion of the monoidal structure when possible (i.e. when the isomorphism classes in the category is a set).

The paper is organized as follows. Section 2 is devoted to review main definitions used in the paper: linear system, feedback isomorphism,  direct sum of linear systems and dynamic isomorphism. These notions generalize respectively the standard notions of pair of matrices, feedback equivalence, dynamic enlargement and dynamic feedback equivalence. We also define stable isomorphism of linear systems as adequate generalization for our purposes of both feedback and dynamic isomorphism.

Section 3 is the core section of the paper devoted to the stable classification of linear systems. We prove that the pair $$S_R=(\textsl{linear systems, feedback morphims})$$ is a category whose isomorphisms are precisely feedback isomorphisms. Thus feedback classification of linear systems is just given by the classes of isomorphisms $(S_R)^{\text{iso}}$. Reachable systems $A_R$ and locally Brunovsky systems $B_R$ arise as subcategories of $S_R$ equipped with the same homomorphisms: the feedback actions.

We define the operation $\oplus$ on linear systems and show that: $(a)$ 'sum' operation $\oplus$ is both the categorical  product and coproduct in the categories of linear systems; $(b)$ dynamic enlargement of a linear system now arises as the 'sum' of the system with a trivial one; and $(c)$ categories of linear systems equipped with $\oplus$ operation are symmetric monoidal, see \cite{ML} or \cite{W}.

Section 3 concludes with a charaterization of stable equivalence in $B_R$ (locally Brunovsky systems) in terms of first $K$-theory group $K_0(B_R)$ of the category of locally Brunovsky systems.

Finally, section 4 is devoted to compute effectively $K_0(B_R)$ as the Grothendieck group completion of the monoid $(B_R)^{\text{iso}}$ of locally Brunovsky systems up to feedback isomorphisms.  

\section{Stable feedback isomorphisms between linear systems}

Let $R$ be a conmutative ring with $1\neq 0$. In this section we introduce the dynamic and stable feedback isomorphisms of linear systems over $R$.

\begin{definition}[cf. \cite{HS}]
A linear system is a triple $\Sigma=(X,f,B)$ where $X$ is a $R$-module, $f:X\rightarrow X$ an endomorphism and $B\subset X$ a submodule.
\end{definition}

\begin{definition}[cf. \cite{M}]
Two linear systems $\Sigma_{1}=(X_{1},f_{1},B_{2})$ and $\Sigma_{2}=(X_{2},f_{2},B_{2})$ are feedback isomorphic (f.i.) if there exist an isomorphism of $R$-modules between the state-spaces $\phi: X_{1}\cong X_{2}$ such that 
\begin{enumerate}
\item $\phi(B_{1})= B_{2}$
\item $Im(f_{2}\circ\phi-\phi\circ f_{1})\subset B_{2}$
\end{enumerate}
\end{definition}

Recall that a pair of matrices $(A,B)\in R^{n\times n}\times R^{n\times m}$ defines a linear system
\begin{equation}
(A,B)\mapsto \Sigma_{A,B}=(R^{n},A,Im(B))
\end{equation}
Note that this correspondence is neither injective nor surjective. On the other hand, feedback isomorphism of linear systems is a generalization of the feedback equivalence of pairs of matrices $(A,B)$ in the following sense: Suppose that pairs of matrices$(A_{1},B_{1})$ and $(A_{2},B_{2})$ are feedback equivalent, then there exist invertible matrices, $P\in GL_n(R)$ and $Q\in GL_m(R)$ and a matrix $K\in R^{m\times n}$ such that $A_{2}=P(A_{1}+B_{1}K)P^{-1}$ and $B_{2}=PB_{1}Q$. Then it is straightforward to show that the matrix $P$ gives a feedback isomorphism between linear systems $\Sigma_{A_{1},B_{1}}$ and $\Sigma_{A_{2},B_{2}}$.

Two pairs of matrices $(A_{1},B_{1})$ and $(A_{2},B_{2})$ are  dynamic equivalent if the pair of matrices of orderes $(p+n\times p+n)$ and $(p+n\times p+m)$
\begin{equation}
\left(\left(\begin{array}{cc}
\textbf{0} & \textbf{0} \\
\textbf{0} & A_{1}
\end{array}
\right), \left(\begin{array}{cc}
\textbf{1} & \textbf{0} \\
\textbf{0} & B_{1}
\end{array}
\right)\right),\left(\left(\begin{array}{cc}
\textbf{0} & \textbf{0} \\
\textbf{0} & A_{2}
\end{array}
\right), \left(\begin{array}{cc}
\textbf{1} & \textbf{0} \\
\textbf{0} & B_{2}
\end{array}
\right)\right)
\end{equation}
are feedback isomorphic (see \cite{BK}, \cite{HCT}). This is physically realized by introducing free ancillary variables.

Consider the pair of matrices $(A,B)$ and let $\Sigma_{A,B}=(R^{n},A,Im(B))$ be the corresponding linear system. Consider also the linear system $\Gamma(p)=(R^{p},0,R^{p})$. Then the linear system associated to the pair of matrices 
\begin{equation}
\left(\left(\begin{array}{cc}
\textbf{0} & \textbf{0} \\
\textbf{0} & A
\end{array}
\right), \left(\begin{array}{cc}
\textbf{1} & \textbf{0} \\
\textbf{0} & B
\end{array}
\right)\right)
\end{equation}
is precisely $(R^{p}\oplus R^{n},0\oplus A,R^{p}\oplus Im(B))$. This motivates the following definition

\begin{definition}
Let $\Sigma_{i}=(X_{i},f_{i},B_{i})$ ($i=1,2$) be linear systems. The direct sum of $\Sigma_{1}$ and $\Sigma_{2}$ is defined by linear system
\begin{equation}
\Sigma_{1}\oplus\Sigma_{2}=(X_{1}\oplus X_{2},f_{1}\oplus f_{2},B_{1}\oplus B_{2})
\end{equation}
\end{definition}

Throughout the paper, we will use Bass matrix notation for the direct sum of homomorphisms (see \cite{Bass}), thus the matrix {\tiny$\left(
\begin{array}{cc}
f_{1} & \textbf{0}\\
\textbf{0} & f_{2}
\end{array}
\right)$}
actually represents the homomorphism $f_{1}\oplus f_{2}$. Elements of the direct sum of objects, $X_{1}\oplus X_{2}$, will be presented as column vectors ${\tiny\left(
\begin{array}{c}
x_{1}\\
y_{1}
\end{array}
\right)}\in X_{1}\oplus X_{2}$ in order to make the notations consistent.

\begin{definition}
Linear systems $\Sigma_{1}$ and $\Sigma_{2}$ are dynamic feedback isomorphic (d.i.) if there exist $p\in\mathbb{N}$ such that the linear sistems $\Sigma_{1}\oplus\Gamma(p)$ and $\Sigma_{2}\oplus\Gamma(p)$ are feedback isomorphic.
\end{definition} 

\begin{definition}
Linear systems $\Sigma_{1}$ and $\Sigma_{2}$ are stable feedback isomorphic (s.i.) if there exist a linear system $\Gamma$ such that the linear sistems $\Sigma_{1}\oplus\Gamma$ and $\Sigma_{2}\oplus\Gamma$ are feedback isomorphic.
\end{definition}

Of course, the relations f.i., d.i, and s.i. satisfy the axioms for equivalence relations in the category of linear systems.
It is also clear that s.i. is a generalization of d.i. and that the d.i. is a generalization of the f.i..

\begin{equation} 
\Sigma\overset{f.e}{\simeq}\Sigma'\Rightarrow\Sigma\overset{d.e}{\simeq}
\Sigma'\Rightarrow\Sigma\overset{s.e}{\simeq}\Sigma'
\end{equation}
Moreover, if $R=\mathbb{K}$ is a field then it is easy to prove that the three relations are equivalent.
\begin{equation} 
\Sigma\overset{f.i}{\simeq}\Sigma'\Leftrightarrow\Sigma\overset{d.i}{\simeq}
\Sigma'\Leftrightarrow\Sigma\overset{s.i}{\simeq}\Sigma'
\end{equation}

\section{Stable classification of locally Brunovsky linear systems}

This section deals with the classification of linear systems modulo stable feedback isomorphisms. Invariants will be found in some group by using a bit of $K$-theory, thus we need to start with the categorical properties of linear systems.

\subsection{The category of linear systems and its monoidal structure}

In order to construct the category of linear systems we need to define the homomorphisms of the category \cite{ML}.

\begin{definition}\label{DefMorSys}
Let $\Sigma_{1}=(X_{1},f_{1},B_{1})$ and $\Sigma_{2}=(X_{2},f_{2},B_{2})$ be linear systems. A homomorphism between $\Sigma_{1}$ and $\Sigma_{2}$ is a homomorphism of $R$-modules $\phi:X_{1}\rightarrow X_{2}$ such that
\begin{enumerate}
\item $\phi(B_{1})\subset B_{2}$
\item $Im(f_{2}\circ\phi-\phi\circ f_{1})\subset B_{2}$
\end{enumerate}
\end{definition}

\begin{lemma}
The composition law for homomorphisms between space states gives the composition law for homomorphisms between linear systems.
$$\begin{array}{ccc}
Hom(\Sigma_{1},\Sigma_{2})\times Hom(\Sigma_{2},\Sigma_{3}) & \rightarrow & Hom(\Sigma_{1},\Sigma_{3}) \\
(\varphi,\psi) & \longmapsto & \psi\circ \varphi
\end{array}$$
satisfying the associative and identity properties.
\end{lemma}
\begin{proof}
Consider homomorphisms $\Phi_{1}\in Hom(\Sigma_{1},\Sigma_{2})$ and $\Phi_{2}\in Hom(\Sigma_{2},\Sigma_{3})$. Both $\Phi_{1}:X_{1}\rightarrow X_{2}$ and $\Phi_{2}:X_{2}\rightarrow X_{3}$ are homomorphisms of $R$-modules between the state spaces. 

Consider also the homomorphism of $R$-modules $\Phi_{2}\circ\Phi_{1}:X_{1}\rightarrow X_{3}$. By hypothesis $\Phi_{2}(B_{2})\subset B_{3}$ and $\Phi_{1}(B_{1})\subset B_{2}$, hence it is clear that $\Phi_{2}(\Phi_{1}(B_{1}))\subset B_{3}$.

Let us show that $Im(f_{3}\circ(\Phi_{2}\circ\Phi_{1})-(\Phi_{2}\circ\Phi_{1})\circ f_{1})\subset B_{3}$. Let be $x_{1}\in X_{1}$, then there exist $b_{2}\in B_{2}$ such that $\phi_{1}(f_{1}(x_{1}))=f_{2}(\phi_{1}(x_{1}))+b_{2}$ from which it follows that $\phi_{2}(\phi_{1}(f_{1}(x_{1})))=\phi_{2}(f_{2}(\phi_{1}(x_{1})))+\phi_{2}(b_{2})$. Finaly, there exist $b_{3}\in B_{3}$ such that $\phi_{2}(f_{2}(\phi_{1}(x_{1})))+\phi_{2}(b_{2})=f_{3}(\phi_{2}(\phi_{1}(x_{1})))+\phi_{2}(b_{2})+b_{3}$ and since $\phi_{2}(b_{2})\in B_{3}$ the result is proved \end{proof}

\begin{definition}
Let $S_{R}$ denotes the category whose objects are linear systems and the homomorphisms are homomorphisms of linear systems.
\end{definition}

The following Proposition shows that the category $S_{R}$ is the natural framework to address the problem of the feedback classification of linear systems.

\begin{theorem}\label{Risomor}
Isomorphisms in the category $S_{R}$ are exactly the feedback isomorphisms.
\end{theorem}
\begin{proof}
Let $\Sigma_{1}$ and $\Sigma_{2}$ linear systems such that $Isom(\Sigma_{1},\Sigma_{2})\neq\emptyset$ and let us take $\phi\in Isom(\Sigma_{1},\Sigma_{2})$. Then there exist $\psi\in Hom(\Sigma_{2},\Sigma_{1})$ such that $\psi\circ\phi=id_{1}$ and $\phi\circ\psi=id_{2}$. This show in particular that the homomorphisms of $R$-modules $\phi:X_{1}\rightarrow X_{2}$ and $\psi:X_{2}\rightarrow X_{1}$ are inverses of each other and, therefore, isomorphims. Now, since $\psi\in Hom(\Sigma_{2},\Sigma_{1})$, we have $\psi(B_{2})\subset B_{1}$. If we take the image via $\phi$ we have $B_{2}\subset\phi(B_{1})$. But $\phi(B_{1})\subset B_{2}$ because $\phi\in Hom(\Sigma_{1},\Sigma_{2})$ so $\phi(B_{1})=B_{2}$ and $\phi$ is a feedback equivalence.

Let us prove the converse. Suppose that $\Sigma_{1}$ and $\Sigma_{2}$ are feedback isomorphic linear systems. Then there exist an isomorphism of $R$-modules $\phi:X_{1}\simeq X_{2}$ such that $\phi(B_{1})=B_{2}$ and $Im(f_{2}\circ\phi-\phi\circ f_{1})\subset B_{2}$. Observe that $\phi$ define a homomorphism of linear systems. In order to prove that $\phi$ is an isomorphism of linear systems we have to find an inverse. Consider the isomorphism of $R$-modules $\phi^{-1}=\psi:X_{2}\rightarrow X_{1}$. It is clear that $\psi(B_{2})=B_{1}$. Let us see that $Im(f_{1}\circ\psi-\psi\circ f_{2})\subset B_{1}$. Consider $x_{2}\in X_{2}$
\begin{equation}
\begin{aligned}
(f_{1}\circ\psi)(x_{2})-(\psi\circ f_{2})(x_{2})\in B_{1}&\Leftrightarrow (f_{1}\circ\psi)(x_{2})-(\psi\circ f_{2})(x_{2})\in \psi(B_{2})\Leftrightarrow \\
&\Leftrightarrow \phi((f_{1}\circ\psi)(x_{2})-(\psi\circ f_{2})(x_{2}))\in B_{2}\Leftrightarrow \\
&\Leftrightarrow (\phi\circ f_{1})(\psi(x_{2}))-f_{2}(x_{2})\in B_{2}\Leftrightarrow \\
&\Leftrightarrow (f_{2}\circ\phi)(\psi(x_{2}))+b_{2}-f_{2}(x_{2})\in B_{2}
\end{aligned}
\end{equation}
for certain $b_{2}\in B_{2}$. But $(f_{2}\circ\phi)(\psi(x_{2}))+b_{2}-f_{2}(x_{2})=b_{2}\in B_{2}$. This prove that $\psi$ define a homomorphism of linear systems and it is clear that $\phi\circ\psi=id_{\Sigma_{2}}$ and that $\psi\circ\phi=id_{\Sigma_{1}}$ as homomorphisms of linear systems.
\end{proof}

Once the category of linear systems is already introduced it is natural to research the algebraic structure induced by the direct sum of linear systems $\oplus$ in this category. We will describe this structure and we will see that $\oplus$ descends to the isomorphism classes of linear systems.

\begin{lemma}
The direct sum of two homomorphisms of linear systems is a homomorphism of linear systems.
\end{lemma}
\begin{proof}
Consider the linear systems $\Sigma_{i}=(X_{i},f_{i},B_{i})$ and $\Gamma_{i}=(Y_{i},g_{i},C_{i})$ (i=1,2) and homomorphisms $\phi:\Sigma_{1}\rightarrow \Sigma_{2}$, $\psi:\Gamma_{1}\rightarrow \Gamma_{2}$. Let us take the direct sum 
\begin{equation}
\phi\oplus\psi={\tiny \left(\begin{array}{cc}\phi & \mathbf{0} \\ \mathbf{0} & \psi\end{array}\right)}:X_1\oplus Y_1\rightarrow X_2\oplus Y_2
\end{equation}
In order to check that $\phi\oplus\psi$ is also a morphism of linear systems (see Definition \ref{DefMorSys}) we need to prove $(1)$ $(\phi\oplus\psi)(B_1\oplus C_1)\subset B_2\oplus C_2$, which is clear; and $(2)$
\begin{equation}
\mathrm{Im}\left((f_2\oplus g_2)(\phi\oplus\psi)-(\phi\oplus\psi)(f_1\oplus g_1)\right)\subset B_2\oplus C_2
\end{equation}
which is a straightforward calculation by using Bass' matrices \end{proof}

The direct sum of linear systems defines a bifunctor $\oplus: S_{R}\times S_{R}\rightarrow S_{R}$. We also define the zero system as $Z=(0,0,0)$. It is clear that $Z$ satisfy the identity property for the direct sum of linear systems. Obviously $Z$ is both an initial and final object in $S_{R}$. Now we can consider the category of linear systems with extra structure $(S_{R},\oplus,Z)$.

\begin{lemma}
The direct sum of linear systems $\oplus$ is both a categorical product and co-product in $S_{R}$, that is to say $\oplus$ is a bi-product in the category $S_{R}$.
\end{lemma}
\begin{proof}
The universal property of product  (see \cite{ML}) arises from the following picture: Given a system $\Gamma$ and homomorphisms $\psi_i:\Gamma\rightarrow\Sigma_i$, dotted line $\Gamma\rightarrow\Sigma_1\oplus\Sigma_2$ always exists making the following diagram commutative and it is given by the adequate Bass' matrix whose entries are the $\psi_i$
\vspace{0.5cm}
\begin{equation}
\begin{xy}
(90,20)*+{\Sigma_1}="Si1";
(90,0)*+{\Sigma_2}="Si2";
(60,10)*+{\Sigma_1\oplus\Sigma_2}="Si1+Si2";
(30,10)*+{\Gamma}="Ga";
{\ar@{->}^{\pi_1} "Si1+Si2";"Si1"};
{\ar@{->}^{\pi_2} "Si1+Si2";"Si2"};
{\ar@{->}@/^{1.85pc}/^{\psi_1} "Ga"; "Si1"}
{\ar@{->}@/_{1.25pc}/_{\psi_2} "Ga"; "Si2"}
{\ar@{-->}^{{\tiny \left(\begin{array}{c}\psi_1 \\ \psi_2\end{array}\right)}} "Ga"; "Si1+Si2"}
\end{xy}
\end{equation}
\vspace{0.25cm}
One can check the universal property of co-product: Given a system $\Gamma$ and homomorphisms $\phi_i:\Sigma_i\rightarrow\Gamma$, dotted line $\Sigma_1\oplus\Sigma_2\rightarrow\Gamma$ always exists making the following diagram commutative and it is given by the adequate Bass' matrix whose entries are the $\phi_i$
\begin{equation}
\begin{xy}
(60,20)*+{\Sigma_1}="Si1";
(60,0)*+{\Sigma_2}="Si2";
(90,10)*+{\Sigma_1\oplus\Sigma_2}="Si1+Si2";
(120,10)*+{\Gamma}="Ga";
{\ar@{->}^{\iota_1} "Si1";"Si1+Si2"};
{\ar@{->}^{\iota_2} "Si2";"Si1+Si2"};
{\ar@{->}@/^{1.25pc}/^{\phi_1} "Si1"; "Ga"}
{\ar@{->}@/_{1.25pc}/_{\phi_2} "Si2"; "Ga"}
{\ar@{-->}^{(\phi_1,\phi_2)} "Si1+Si2"; "Ga"}
\end{xy}
\end{equation}
\end{proof}

\begin{theorem} Let $R$ be a commutative ring. The category $(S_R,\oplus,Z)$ of linear systems is symmetric monoidal
\end{theorem}
\begin{proof}
It is a direct consequence of $\oplus$ being biproduct and that $Z$ is both initial and terminal, see \cite{ML} or \cite{W}) \end{proof}

\begin{remark}$(i)$ It is worth to note that the direct sum in $S_{R}$ descends to the isomorphism classes of linear systems. To be precise, if $\Sigma_{1}\cong\Sigma_{2}$ and $\Gamma_{1}\cong\Gamma_{2}$ then $\Sigma_{1}\oplus\Gamma_{1}\cong\Sigma_{2}\oplus\Gamma_{2}$. $(ii)$ On the other hand note that if $\Sigma_1\oplus\Sigma_2\cong\Sigma_2\oplus\Sigma_1$ by means of ${\tiny \left(\begin{array}{cc}\mathbf{0} & \mathbf{1} \\\mathbf{1} & \mathbf{0}\end{array}\right)}:X_1\oplus X_2\rightarrow X_2\oplus X_1$
\end{remark}

\subsection{Stable classification of locally Brunovsky linear systems and the $K_{0}$ group}

Once we have obtained the symmetric monoidal structure of categories of linear systems and feedback actions we will characterize the stable isomorphism of linear systems in terms of the $0$-th $K$-theory group of the category. The construction of that group \cite{W} is to complete the monoid of isomorphism classes $(S_R)^{\text{iso}}$. But in general $(S_R)^{\text{iso}}$ is not even a set. To avoid this obstruction, the subcategory of locally Brunovsky linear systems is considered because, in this case, the isomorphisms classes form a well defined set, in fact $B_R^{\text{iso}}=\mathbf{P}(R)^{\infty}$ is the set of finite support sequences of finitely generated projective $R$-modules \cite{M} .

Let $\Sigma=(X,f,B)$ be a linear system over the ring $R$. Recall the definition of the invariant modules associated to $\Sigma$ see \cite{M}:
\begin{enumerate}
\item $N_{i}=B+f(N_{i-1})$ for $i\geq 1$ being $N_{0}=0$
\item $M_{i}=X/N_{i}$
\item $I_{i}=ker(M_{i-1}\overset{1}{\rightarrow}M_{i}\rightarrow 0)$ (I-invariants)
\item $Z_{i}=ker(I_{i}\overset{f}{\rightarrow}I_{i+1}\rightarrow 0)$ (Z-invariants)
\end{enumerate}

\begin{lemma}\label{directsuminv}
Let $\Sigma_{1}$ and $\Sigma_{2}$ be linear systems and consider the direct sum $\Sigma_{1}\oplus\Sigma_{2}$. Then
\begin{enumerate}
\item $N_{i}^{\Sigma_{1}\oplus\Sigma_{2}}= N_{i}^{\Sigma_{1}}\oplus N_{i}^{\Sigma_{2}}$
\item $M_{i}^{\Sigma_{1}\oplus\Sigma_{2}}\cong M_{i}^{\Sigma_{1}}\oplus M_{i}^{\Sigma_{2}}$
\item $I_{i}^{\Sigma_{1}\oplus\Sigma_{2}}\cong I_{i}^{\Sigma_{1}}\oplus I_{i}^{\Sigma_{2}}$
\item $Z_{i}^{\Sigma_{1}\oplus\Sigma_{2}}\cong Z_{i}^{\Sigma_{1}}\oplus Z_{i}^{\Sigma_{2}}$
\end{enumerate}
\end{lemma}
\begin{proof}
Let's denote by columns the elements of $X=X_1\oplus X_2$. Hence homomorphism $f_1\oplus f_2$ is, in Bass' notation {\tiny$\left(\begin{array}{cc}f_1 & \textbf{0} \\\textbf{0} & f_2\end{array}\right)$}.

$(1)$ Is clear because $N_0^{\Sigma_1\oplus\Sigma_2}=(0,0)$, $N_1^{\Sigma_1\oplus\Sigma_2}=B_1\oplus B_2=N_{1}^{\Sigma_{1}}\oplus N_{2}^{\Sigma_{2}}$, and one obtains sequently $N_i^{\Sigma_1\oplus\Sigma_2}=(B_1\oplus B_2)+ {\tiny \left(\begin{array}{cc}f_1 & \textbf{0} \\ \textbf{0} & f_2\end{array}\right)}(N_{i-1}^{\Sigma_1}\oplus N_{i-1}^{\Sigma_2})=N_{i}^{\Sigma_{1}}\oplus N_{i}^{\Sigma_{2}}$.

$(2)$ Since $M_i^{\Sigma_1\oplus\Sigma_2}= X/N_i^{\Sigma_1\oplus\Sigma_2}$, its elements are the classes $(x_{1},x_{2})+N_i^{\Sigma_1\oplus\Sigma_2}$. Consider the (well defined) linear map $\mu(x_1+N_i^{\Sigma_1},x_2+N_i^{\Sigma_2})=
(x_1,x_2)+N_i^{\Sigma_1\oplus\Sigma_2}$. Since  $\pi(x_{1},x_{2})=(x_{1},x_{2})+
N_i^{\Sigma_1\oplus\Sigma_2}$, then the result follows from application of short-five-lemma on the following commutative diagram with exact rows

\vspace{0.5cm}
$$
\begin{xy}
(30,0)*+{0}="11";
(60,0)*+{N_i^{\Sigma_1\oplus\Sigma_2}}="12";
(90,0)*+{X}="13";
(120,0)*+{M_i^{\Sigma_1\oplus\Sigma_2}}="14";
(150,0)*+{0}="15";
(30,20)*+{0}="21";
(60,20)*+{N_i^{\Sigma_1}\oplus N_i^{\Sigma_2}}="22";
(90,20)*+{X_1\oplus X_2}="23";
(120,20)*+{M_i^{\Sigma_1}\oplus M_i^{\Sigma_2}}="24";
(150,20)*+{0}="25";
(105,10)*+{\circlearrowleft}="a";
(75,10)*+{\circlearrowleft}="b";
{\ar@{->}^{} "21";"22"};
{\ar@{->}^{{\tiny\left(
\begin{array}{cc}
i_{1} & \textbf{0}\\
\textbf{0} & i_{2}
\end{array}
\right)}} "22";"23"};
{\ar@{->}^{{\tiny\left(
\begin{array}{cc}
\pi_{1} & \textbf{0}\\
\textbf{0} & \pi_{2}
\end{array}
\right)}} "23";"24"};
{\ar@{->}^{} "24";"25"};
{\ar@{->}^{} "11";"12"};
{\ar@{->}^{i} "12";"13"};
{\ar@{->}^{\pi} "13";"14"};
{\ar@{->}^{} "14";"15"};
{\ar@{=}^{} "12";"22"};
{\ar@{=}^{} "13";"23"};
{\ar@{->}^{\mu} "24";"14"};
\end{xy}
$$
\vspace{0.5cm}

$(3)$ Let $\mu$ be the linear maps defined in $(2)$. It is clear that the following square is commutative
$$
\xymatrix{
0\ar[r] & I_{i}^{\Sigma_{1}}\oplus I_{i}^{\Sigma_{2}}\ar[r]^{{\tiny\left(
\begin{array}{cc}
i_{1} & \textbf{0}\\
\textbf{0} & i_{2}
\end{array}
\right)}}  & M_{i-1}^{\Sigma_{1}}\oplus M_{i-1}^{\Sigma_{2}}\ar[r]^{{\tiny\left(
\begin{array}{cc}
1 & \textbf{0}\\
\textbf{0} & 1
\end{array}
\right)}}\ar[d]^{\mu}\ar@{}[rd]|{\circlearrowleft} & M_{i}^{\Sigma_{1}}\oplus M_{i}^{\Sigma_{2}} 
\ar[d]^{\mu}\ar[r] & 0 \\
0\ar[r] & I_{i}^{\Sigma_{1}\oplus\Sigma_{2}}\ar[r]^{i}  & M_{i-1}^{\Sigma_{1}\oplus\Sigma_{2}}\ar[r]^{1} & M_{i}^{\Sigma_{1}\oplus\Sigma_{2}}\ar[r] & 0
}
$$
Since $1\circ \mu \circ{\tiny\left(
\begin{array}{cc}
i_{1} & \textbf{0}\\
\textbf{0} & i_{2}
\end{array}
\right)}=\mu\circ{\tiny\left(
\begin{array}{cc}
1 & \textbf{0}\\
\textbf{0} & 1
\end{array}
\right)}\circ {\tiny\left(
\begin{array}{cc}
i_{1} & \textbf{0}\\
\textbf{0} & i_{2}
\end{array}
\right)}=0$ we deduce that the image of $\mu\circ {\tiny\left(
\begin{array}{cc}
i_{1} & \textbf{0}\\
\textbf{0} & i_{2}
\end{array}
\right)}$ lies into $I_{i}^{\Sigma_{1}\oplus\Sigma_{2}}$. Define $\nu$ as the restriction of $\mu$ to $I_{i}^{\Sigma_{1}}\oplus I_{i}^{\Sigma_{2}}$ ($\nu(x_1+N_{i-1}^{\Sigma_1},x_2+N_{i-1}^{\Sigma_2})=
\mu(x_1+N_{i-1}^{\Sigma_1},x_2+N_{i-1}^{\Sigma_2})$), then the result follows from application of short-five-lemma on the following commutative diagram with exact rows
\vspace{0.5cm}
$$
\begin{xy}
(30,0)*+{0}="11";
(60,0)*+{I_i^{\Sigma_1\oplus\Sigma_2}}="12";
(90,0)*+{M_{i-1}^{\Sigma_1\oplus\Sigma_2}}="13";
(120,0)*+{M_i^{\Sigma_1\oplus\Sigma_2}}="14";
(150,0)*+{0}="15";
(30,20)*+{0}="21";
(60,20)*+{I_i^{\Sigma_1}\oplus I_i^{\Sigma_2}}="22";
(90,20)*+{M_{i-1}^{\Sigma_1}\oplus M_{i-1}^{\Sigma_2}}="23";
(120,20)*+{M_i^{\Sigma_1}\oplus M_i^{\Sigma_2}}="24";
(150,20)*+{0}="25";
(105,10)*+{\circlearrowleft}="a";
(75,10)*+{\circlearrowleft}="b";
{\ar@{->}^{} "21";"22"};
{\ar@{->}^{{\tiny\left(
\begin{array}{cc}
i_{1} & \textbf{0}\\
\textbf{0} & i_{2}
\end{array}
\right)}} "22";"23"};
{\ar@{->}^{{\tiny\left(
\begin{array}{cc}
1 & \textbf{0}\\
\textbf{0} & 1
\end{array}
\right)}} "23";"24"};
{\ar@{->}^{} "24";"25"};
{\ar@{->}^{} "11";"12"};
{\ar@{->}^{i} "12";"13"};
{\ar@{->}^{1} "13";"14"};
{\ar@{->}^{} "14";"15"};
{\ar@{->}^{\nu} "22";"12"};
{\ar@{->}^{\mu} "23";"13"};
{\ar@{->}^{\mu} "24";"14"};
\end{xy}
$$
\vspace{0.5cm}

$(4)$ As above, ${\tiny\left(
\begin{array}{cc}
f_{1} & \textbf{0}\\
\textbf{0} & f_{2}
\end{array}
\right)}\circ\nu\circ{\tiny\left(
\begin{array}{cc}
i_{1} & \textbf{0}\\
\textbf{0} & i_{2}
\end{array}
\right)}=\nu\circ{\tiny\left(
\begin{array}{cc}
f_{1} & \textbf{0}\\
\textbf{0} & f_{2}
\end{array}
\right)}\circ{\tiny\left(
\begin{array}{cc}
i_{1} & \textbf{0}\\
\textbf{0} & i_{2}
\end{array}
\right)}=0$ so $\nu\circ{\tiny\left(
\begin{array}{cc}
i_{1} & \textbf{0}\\
\textbf{0} & i_{2}
\end{array}
\right)}$ lies into $Z_{i}^{\Sigma_{1}\oplus\Sigma_{2}}$. Defining $\rho$ as the restriction of $\nu$ to $Z_{i}^{\Sigma_{1}}\oplus Z_{i}^{\Sigma_{2}}$ ($\rho(x_1+N_{i-1}^{\Sigma_1},x_2+N_{i-1}^{\Sigma_2}) =\nu(x_1+N_{i-1}^{\Sigma_1},x_2+N_{i-1}^{\Sigma_2})$) we see that the following diagram is commutative

\vspace{0.5cm}
$$
\begin{xy}
(30,0)*+{0}="11";
(60,0)*+{Z_i^{\Sigma_1\oplus\Sigma_2}}="12";
(90,0)*+{I_{i}^{\Sigma_1\oplus\Sigma_2}}="13";
(120,0)*+{I_{i+1}^{\Sigma_1\oplus\Sigma_2}}="14";
(150,0)*+{0}="15";
(30,20)*+{0}="21";
(60,20)*+{Z_i^{\Sigma_1}\oplus Z_i^{\Sigma_2}}="22";
(90,20)*+{I_{i}^{\Sigma_1}\oplus I_{i}^{\Sigma_2}}="23";
(120,20)*+{I_{i+1}^{\Sigma_1}\oplus I_{i+1}^{\Sigma_2}}="24";
(150,20)*+{0}="25";
(105,10)*+{\circlearrowleft}="a";
(75,10)*+{\circlearrowleft}="b";
{\ar@{->}^{} "21";"22"};
{\ar@{->}^{{\tiny\left(
\begin{array}{cc}
i_{1} & 0\\
0 & i_{2}
\end{array}
\right)}} "22";"23"};
{\ar@{->}^{{\tiny\left(
\begin{array}{cc}
f_{1} & 0\\
0 & f_{2}
\end{array}
\right)}} "23";"24"};
{\ar@{->}^{} "24";"25"};
{\ar@{->}^{} "11";"12"};
{\ar@{->}^{i} "12";"13"};
{\ar@{->}^{{\tiny\left(
\begin{array}{cc}
f_{1} & \textbf{0}\\
\textbf{0} & f_{2}
\end{array}
\right)}} "13";"14"};
{\ar@{->}^{} "14";"15"};
{\ar@{->}^{\rho} "22";"12"};
{\ar@{->}^{\nu} "23";"13"};
{\ar@{->}^{\nu} "24";"14"};
\end{xy}
$$
\vspace{0.5cm}

Then the result follows from application of short-five-lemma on the above diagram with exact rows. 
\end{proof}

A linear system $\Sigma$ is reachable if $N_{s}=X$ or, equivalently, if $M_{s}=0$ (see \cite{HS} and \cite{M}).
$\Sigma$ is a locally Brunovsky linear system if the state space is finitely generated and the invariant modules are projective $R$-modules. Show that, in particular, locally Brunovsky linear systems are reachable (see \cite{M}).

Let $B_{R}$ the subcategory of $S_{R}$ whose objects are the locally Brunovsky linear systems and whose homomorphisms are the homomorphisms of linear systems. Since direct sum of finitelly generated projectives is again projective it follows that direct sum of locally Brunovsky linear systems is again a locally Brunovsky linear system.

In fact, $B_{R}$ is a symmetric monoidal subcategory of $S_{R}$. Moreover, since the isomorphism classes of locally Brunovsky linear systems, $B_{R}^{iso}$, is a set (see \cite{M}) we have

\begin{lemma}
The triple $(B_{R}^{iso},\oplus,0)$ is a commutative monoid.
\end{lemma}
\begin{proof}
 Since the direct sum descends to the isomorphisms classes of linear systems, $\oplus$ is well defined in $B_{R}^{iso}$ (that is, $\oplus$ is a closed binary operation in $B_{R}^{iso}$). The identity element is the class of the zero linear system $Z=(0,0,0)$ and the associativity and commutativity properties follows from the associativity and commutativity of $\oplus$ in the symmetric monoidal category $B_{R}$.
\end{proof}

Being the set of isomorphism classes of locally Brunovsky linear systems is conmutative monoid, it is natural to ask for the relationship between its Grothendieck (completion) group and the theory of linear systems. The next Theorem shows the close link between the Grothendieck group of the monoid $B_{R}^{iso}$ and the set of stable equivalence classes of locally Brunovsky linear systems over $R$.

\begin{theorem}
Let us denote $K_{0}(B_{R})$ the Grothendieck group of the monoid $B_{R}^{iso}$ and $\gamma:B_{R}^{iso}\rightarrow K_{0}(B_{R})$ the natural homomorphism of monoids. Then $\gamma([\Sigma])=\gamma([\Gamma])$ if and only if $\Sigma\overset{s.i.}{\simeq}\Gamma$.
\end{theorem}
\begin{proof}
The Grothendieck group of $B_{R}^{iso}$ is 
$K_{0}(B_{R})=(B_{R}^{iso}\times B_{R}^{iso})/\sim$
being the equivalence relation $\sim$ as follows 
\begin{equation}
\begin{aligned}
&([\Sigma_{1}],[\Gamma_{1}])\sim([\Sigma_{2}],[\Gamma_{2}])\Leftrightarrow \exists [U]\in B_{R}^{iso} \\
&\text{such that } [\Sigma_{1}]+[\Gamma_{2}]+[U]=[\Sigma_{2}]+[\Gamma_{1}]+[U]
\end{aligned}
\end{equation}
We will denote $<\Sigma,\Gamma>$ the class of $([\Sigma],[\Gamma])$ in $K_{0}(B_{R})$. Then the natural homomorphism of monoids $\gamma$ is defined by $\gamma([\Sigma])=<\Sigma,0>$.

Now suppose that $\gamma([\Sigma_{1}])=\gamma([\Sigma_{2}])$, then $<\Sigma_{1},0>=<\Sigma_{2},0>$ and by definition there exist a linear system $U$ such that $[\Sigma_{1}]+[U]=[\Sigma_{2}]+[U]$. But the equalities in $B_{R}^{iso}$ are the feedback isomorphisms in the category $B_{R}$, so
\begin{equation*}
\Sigma_{1}\oplus U\overset{f.i.}{\simeq}\Sigma_{2}\oplus U
\end{equation*}
and we deduce that $\Sigma_{1}\overset{s.i.}{\simeq}\Sigma_{2}$.
\end{proof}

\begin{corollary}
The conmutative sub-monoid $Im(\gamma)\subset K_{0}(B_{R})$ is precisely the stable equivalence classes of locally Brunovsky linear systems over $R$. 
\end{corollary}

\section{The $K_{0}$ group of locally Brunovsky linear systems}

It have been proved in \cite{M} that there exist a bijective correspondence between the feedback isomorphism classes of locally Brunovsky linear systems over the ring $R$ and the set, $\mathbb{P}(R)^{\infty}$, of finite support sequences with entries in $\mathbb{P}(R)$. This correspondence is given by the map of the Z-invariants
\begin{equation}
\begin{aligned}
B_{R}^{iso}&\overset{\mathcal{Z}}{\rightarrow} \mathbb{P}(R)^{\infty} \\
[\Sigma]&\mapsto (Z_{1}^{\Sigma},Z_{2}^{\Sigma},\hdots,Z_{s}^{\Sigma},0,0,\hdots)
\end{aligned}
\end{equation}

Observe that $\mathbb{P}(R)^{\infty}$ have a monoid structure given by the direct sum of sequences and that, as a monoid, is isomorphic to $\bigoplus_{\mathbb{N}}\mathbb{P}(R)$. Then, from Lemma \ref{directsuminv}, follows that $\mathcal{Z}$ is an isomorphism of monoids. This allow us to give a precise description of the monoid of stable equivalence classes of locally Brunovsky linear systems.

\begin{theorem}\label{mainteo2} Let $\mathbb{P}(R)$ be the monoid of isomorphisms classes of projective finitely generated $R$-modules and $K_{0}(R)$ the Grothendieck group of the monoid $\mathbb{P}(R)$. Then
\begin{enumerate}
\item $B_{R}^{iso}\simeq\bigoplus_{\mathbb{N}}\mathbb{P}(R)$
\item $K_{0}(B_{R})\simeq\bigoplus_{\mathbb{N}}K_{0}(R)$
\end{enumerate}
\end{theorem}
\begin{proof}
\begin{enumerate}
\item It is clear combining the biyection given by the $Z$-map defined in \cite{M} and the Lemma \ref{directsuminv}.

\item Recall that if $M$ is a monoid then $K_{0}(M)$ has the universal property (see \cite{R}): $K_{0}(M)$ is the unique abelian group (up to isomorphisms) such that for any other abelian group $G$ and any monoid homomorphism $g:M\rightarrow G$ there exist a unique group homomorphism $f: K_{0}(M)\rightarrow G$ such that the diagram is commutative
$$
\xymatrix{
 M\ar[rr]^{g}\ar[d]^{\gamma} &  & G           \\
K_{0}(M)\ar@{-->}[urr]^{f}  &  &
}
$$
being $\gamma$ the completion homomorphism $\gamma:M\rightarrow K_{0}(M)$.

Let $\gamma:\mathbb{P}(R)\rightarrow K_{0}(R)$ be the completion homomorphism and consider the induced monoid homomorphism $\overline{\gamma}:\bigoplus_{\mathbb{N}}\mathbb{P}(R)\rightarrow \bigoplus_{\mathbb{N}}K_{0}(R)$. Consider an abelian group and a homomorphism of monoids $\bigoplus_{\mathbb{N}}\mathbb{P}(R)\rightarrow G$. Then, by the universal property of the direct sum, there is a family of monoid homomorphisms $g_{i}:\mathbb{P}(R)\rightarrow G$ such that $g=\oplus g_{i}$. Because of the universal property of $K_{0}(R)$ there exist, for each $g_{i}$, a unique group homomorphism, 
$f_{i}: K_{0}(R)\rightarrow G$, making the following diagram commutative
$$
\xymatrix{
\mathbb{P}(R) \ar[rr]^{g_{i}}\ar[d]^{\gamma} &  & G           \\
K_{0}(R)\ar@{-->}[urr]^{f_{i}}  &  &
}
$$
Now, it is clear that the following diagram is commutative
$$
\xymatrix{
\bigoplus_{\mathbb{N}}\mathbb{P}(R) \ar[rr]^{g}\ar[d]^{\overline{\gamma}} &  & G           \\
\bigoplus_{\mathbb{N}}K_{0}(R)\ar@{-->}[urr]^{\oplus f_{i}}  &  &
}
$$
wich proves the existence. If there is another group homomorphism $f':\bigoplus_{\mathbb{N}}K_{0}(R)\rightarrow G$ making the diagram commutative then, by the universal property of the direct sum and Grothendieck group, it must be equal to $f$.

\end{enumerate}
\end{proof}

Theorem \ref{mainteo2} reduces the problem of study the monoid of stable isomorphism classes of locally Brunovsky linear sistems to the study of the sub-monoid given by the image of natural map $\gamma:\mathbb{P}(R)\rightarrow K_{0}(R)$ for the base ring

\begin{corollary}\label{main}
The conmutative sub-monoid $\bigoplus_{\mathbb{N}}Im(\alpha)\subset \bigoplus_{\mathbb{N}} K_{0}(R)$ is precisely the stable equivalence classes of locally Brunovsky linear systems over $R$. 
\end{corollary} 

\begin{corollary}
We give a brief summary of main characterizations for locally Brunovsky linear systems:
\begin{enumerate}
\item Two linear systems $\Sigma_i$ in $B_R$ are feedback isomorphic if and only if their images under $Z$-map agree on $\bigoplus_{\mathbb{N}}\mathbb{P}(R)$, i.e. $Z(\Sigma_1)=Z(\Sigma_2)$.
\item Two linear systems $\Sigma_i$ in $B_R$ are stable feedback isomorphic if and only if their images under $Z$-map agree on $K_0(R)^{\infty}$, i.e. $\gamma(Z(\Sigma_1))=\gamma(Z(\Sigma_2))$.

\end{enumerate}
\end{corollary} 

To conclude we review an example of \cite{M}. Let $R=\mathbb{R}[x,y,z]/(x^2+y^2+z^2-1)$ be the coordinate ring of unit sphere $\mathbb{S}^2_{\mathbb{R}}\subseteq\mathbb{R}^3$ immersed into $3$-dimensional space. Consider the state-space $R^4$ and fix the standard basis $\{e_i\}$. Let the linear systems $\Sigma=(R^4,f,B)$ and $\Sigma^{\prime}=(R^4,f^{\prime},B)$; where, in the standard basis,
$$
f=\left(\begin{array}{cccc}0 & 0 & 0 & 0 \\0 & 0 & 0 & 0 \\0 & 0 & 0 & 0 \\1 & 0 & 0 & 0\end{array}\right), \hspace{0.2cm}
f^{\prime}=\left(\begin{array}{cccc}0 & 0 & 0 & 0 \\0 & 0 & 0 & 0 \\0 & 0 & 0 & 0 \\x & y & z & 0\end{array}\right)
$$
It is proven in \cite{M} that above linear systems are not feedback isomorphic because ... but note that they lie in the same class in $K_0(B_R)$ hence they are stable isomorphic. In fact both systems became feedback isomorphic by adding the trivial ancillary system $\Gamma(1)=(R,0,R)$. The reader can check that systems $\Gamma(1)\oplus\Sigma$ and $\Gamma(1)\oplus\Sigma^{\prime}$ are feedback isomorphic by means of the isomorphism of $R^5$ given, in standard basis by
$$
\phi=\left(\begin{array}{ccccc}1 & 0 & 0 & 0 & 0 \\x & 1 & 0 & 0 & 0 \\y & 0 & 1 & 0 & 0 \\z & 0 & 0 & 1 & 0 \\0 & 0 & 0 & 0 & 1\end{array}\right)
$$

\end{document}